\documentclass[12pt]{amsart}

\usepackage{amsmath,amscd,amsfonts,amsthm,amssymb}
\usepackage[shortlabels]{enumitem}
\usepackage{amscd}
\usepackage[all,cmtip]{xy}
\usepackage[normalem]{ulem}
\usepackage[boxsize=8pt]{ytableau}
\usepackage[left=1in,top=1in,right=1in,bottom=1in
]{geometry}
\usepackage{graphicx}
\usepackage{tikz}
\usepackage{stmaryrd}
\usepackage{mathrsfs}
\usepackage{cite}
\usepackage{mathtools}
\usetikzlibrary{matrix,calc}
\usepackage{hyperref}

\newtheorem{thm}{Theorem}[section]
\newtheorem{lem}[thm]{Lemma}
\newtheorem{cor}[thm]{Corollary}

\newtheorem{prop}[thm]{Proposition}
\newtheorem{obs}[thm]{Observation}

\theoremstyle{definition}
\newtheorem{defn}[thm]{Definition}

\theoremstyle{remark}
\newtheorem{rem}[thm]{Remark}

\newtheorem{ex}[thm]{Example}
\newtheorem{cons}[thm]{Construction}

\newcommand{\bi}{\begin{itemize}}
\newcommand{\ei}{\end{itemize}}
\newcommand{\ben}{\begin{enumerate}}
\newcommand{\een}{\end{enumerate}}
\newcommand{\be}{\begin{equation}}
\newcommand{\ee}{\end{equation}}
\newcommand{\bea}{\begin{eqnarray}}
\newcommand{\eea}{\end{eqnarray}}
\newcommand{\bal}{\begin{align}}
\newcommand{\eal}{\end{align}}
\newcommand{\ba}{\begin{array}}
\newcommand{\ea}{\end{array}}

\newcommand{\schur}[1]{\scalebox{1.6}{$\mathbb{S}$}_{\ydiagram{3,2,1,1}}}

\newcommand{\catname}[1]{{\normalfont\textbf{#1}}}
\newcommand{\mop}[1]{{\normalfont\text{#1}}}

\newcommand{\Fin}{\catname{FinSet}}

\newcommand{\RMod}{\catname{Mod$_{k}$}}

\newcommand{\Rep}{\catname{Rep}_k(\mathcal{D})}

\newcommand{\Ch}{\mathcal{C}\hspace{-1pt}h}

\newcommand{\Ob}{\mop{Ob}}
\newcommand{\Hom}{\mop{Hom}}

\newcommand{\Ext}{\mop{Ext}}

\newcommand{\End}{\mop{End}}

\newcommand{\rann}[2]{\prescript{}{#1}{#2}}
\newcommand{\lann}[2]{{#1}_{#2}}

\newcommand{\im}{\mop{Im}}

\newcommand{\Lan}{\mop{Lan}}

\newcommand{\Res}{\mop{Res}}


\title{Categories of Dimension Zero}

\author{John D. Wiltshire-Gordon}
\address{\scriptsize UW-Madison Department of Mathematics;
Van Vleck Hall;
480 Lincoln Drive;
Madison, WI  53706}
\email{jwiltshiregordon@gmail.com}



\date{June 28, 2017}

\begin{document}
\maketitle
\begin{abstract}
If $\mathcal{D}$ is a category and $k$ is a commutative ring, the functors from $\mathcal{D}$ to $\RMod$ can be thought of as representations of $\mathcal{D}$.  By definition, $\mathcal{D}$ is dimension zero over $k$ if its finitely generated representations have finite length.  We characterize categories of dimension zero in terms of the existence of a ``homological modulus'' (Definition \ref{defn:hommod}) which is combinatorial and linear-algebraic in nature.  
\end{abstract}

\section{Introduction}
A ring is called \textbf{dimension zero} if its finitely generated modules have finite length.  A fundamental example is the group algebra $kG$ of a finite group.  Modules for the group algebra are representations of $G$, and being dimension zero means that any representation that can be $G$-spanned by finitely many vectors has a finite composition series.  A group $G$ is called \textbf{dimension zero over a commutative ring $k$} if the ring $kG$ is dimension zero.

It is natural to ask if the condition that $G$ be finite is strictly necessary; perhaps there are some infinite groups that are dimension zero.  There are not, as was shown by Connell:

\begin{thm}[Connell \cite{Connell63}]
A group $G$ is dimension zero over a commutative ring $k$ if and only if $k$ is Artinian and $G$ is finite.
\end{thm}

A generalization of this result due to Zel$'$manov characterizes monoids of dimension zero.  (Recall that a monoid is a group without inverses, or equivalently, a semigroup with identity.)

\begin{thm}[Zel$'$manov \cite{Zelmanov}]
A monoid $M$ is dimension zero over a commutative~ring~$k$ if and only if $k$ is Artinian and $M$ is finite.
\end{thm}

A group, or a monoid, comes with a composition law, whose properties abstract the usual composition of functions.
Specifically, groups model invertible self-functions (permutations) and monoids model arbitrary self-functions (possibly non-invertible).
Without doubt, however, the algebraic structure that best abstracts the notion of function composition is that of a category.

Theorem \ref{thm:detection} contributes a characterization of categories of dimension zero.  Strikingly, and in contrast to the cases of groups and monoids, such categories may be infinite.

\subsection{Representations of a category}
A category has objects and arrows between them.  We think of the arrows as being like the elements of a group or a monoid: they can be composed, the composition law is associative, and there are identities.  Concretely, a representation of a category $\mathcal{D}$ over a field $k$ is an assignment of a rectangular $k$-matrix $V(f)$ to every arrow $f$ in $\mathcal{D}$ so that
$$
f \circ g = h \;\;\; \implies \;\;\; V(f) \circ V(g) = V(h)
$$
and the identity arrows are assigned to identity matrices of various sizes.  Just as a representation of $G$ converts abstract symmetries to linear symmetries, a representation of $\mathcal{D}$ converts abstract transformations to linear transformations.

\subsection{Avatars for the representation theory of a category}
The representation theory of categories is a broad and flexible subject, going by many names, and with many possible focuses.  For example, if $\mathcal{D} = \mathcal{C}^{op}$ then a representation of $\mathcal{D}$ is a presheaf of $k$-vector spaces on $\mathcal{C}$.  Or, taking $\mathcal{D}$ to be the category freely generated by the arrows of a finite directed graph, a representation is the same as a quiver representation.

In commutative algebra, many notions of ``graded module'' may be recast in terms of the representation theory of a category whose objects correspond to grades, and whose arrows correspond to structure maps.  For example, a representation of the poset category $(\mathbb{N}, \leq)$ is the same as a nonnegatively-graded $k[T]$-module.

The homological algebra present in a category of representations goes by the name ``functor homology.''  The topic is expansive, with connections to homotopy theory and the cohomology of classical and algebraic groups.  For an introduction, see \cite{FunctorHomology}.

This author's contact with the subject comes from the representation stability phenomena of Church-Farb and Church-Ellenberg-Farb \cite{ChurchFarbRTHS, CEF}.  In this direction, the prototypical category of study is $\mathcal{D}=\mathrm{FI}$, the category of finite sets and injections.

\subsection{Definitions}
Given a commutative ring $k$ and a category $\mathcal{D}$, a \textbf{representation} of $\mathcal{D}$ is a functor $V : \mathcal{D} \longrightarrow \RMod$ to the category of $k$-modules.  A map between representations $V \longrightarrow W$ is a natural transformation.  Kernels and images of maps are formed objectwise, and so the category of representations $\catname{Rep}_k(\mathcal{D})$ is abelian.

We say $V$ is \textbf{generated in degrees $\{d_i\}_{i \in I}$} for some collection of objects $d_i \in \mathcal{D}$ if any subrepresentation $W \subseteq V$ with $W(d_i) = V(d_i)$ has $W=V$.  We say $V$ is \textbf{finitely generated} if it is generated in degrees $\{d_1, \ldots, d_l\}$ and each $V(d_i)$ is finitely generated as a $k$-module.  Concretely, if $V$ is finitely generated, then there are finitely many $v_i \in Vd_i$ spanning $V$ using scalars from $k$ and arrows from $\mathcal{D}$.  We say $V$ has \textbf{finite length} if it satisfies both the ascending and descending chain conditions on subrepresentations.  A nonempty category $\mathcal{D}$ is \textbf{dimension zero over $k$} if all its finitely generated representations have finite length.

As is typical when studying functor categories, we assume throughout that $\mathcal{D}$ is small.
\subsection{Results}
Given objects $x, y \in \mathcal{D}$, let $S(x,y)$ be the set of self-maps of $x$ that factor through $y$.  For each $d \in \mathcal{D}$ and $s \in S(x,y)$, define a square 0-1-matrix $M_s$ whose rows and columns are indexed by $\Hom_{\mathcal{D}}(d, x)$.  Put a $1$ in position $(f,g)$ whenever $s \circ f = g$.  In other words, the entries of $M_s$ record the commutativity of the following diagram:
$$
\xymatrix{
& d \ar[ld]_f \ar[rd]^g & \\
x \ar[r] \ar@/_1.4pc/[rr]_s & y\ar[r] & x.
}
$$
\begin{defn} \label{defn:preorder}
Write $x \leq_d y$ if the identity matrix is in the $k$-span of the matrices $M_s$.
\end{defn}
\noindent
In Lemma \ref{lem:preorder}, we will see that $\leq_d$ is reflexive and transitive, and so defines a preorder.  An equivalent definition of $\leq_d$ better-suited to proofs is available in Proposition \ref{prop:alt}.

\begin{defn} \label{defn:hommod}
A \textbf{homological modulus} for a category $\mathcal{D}$ over a ring $k$ is a function
$$
\mu : \Ob(\mathcal{D}) \longrightarrow \{ \mbox{finite subsets of } \Ob(\mathcal{D}) \}
$$
such that for every $d, x \in \mathcal{D}$, there exists $y \in \mu(d)$ so that $x \leq_d y$.
\end{defn}

\begin{thm} \label{thm:detection}
A category $\mathcal{D}$ is dimension zero over a commutative ring $k$ if and only if $k$ is Artinian, each hom-set $\Hom_{\mathcal{D}}(x, y)$ is finite, and $\mathcal{D}$ admits a homological modulus over $k$.
\end{thm}

\begin{ex}[An infinite category of dimension zero] \label{ex:dkc}
Define the \textbf{simplex category} $\Delta$ with objects $\{1, \ldots, n\}$ for $n \in \mathbb{N}$, $n \geq 1$, and weakly monotone functions between them; representations of $\Delta$ are usually called  \textbf{cosimplicial $k$-modules}.  Under the Dold-Kan Correspondence (see \cite{doldclassical}, or \cite{weibel} for a modern account), finitely generated cosimplicial $k$-modules correspond to bounded cochain complexes of finitely generated $k$-modules supported in nonnegative degrees.  It follows that $\Delta$ is dimension zero over any Artinian commutative ring.  In keeping with Theorem \ref{thm:detection}, we produce an explicit homological modulus for $\Delta$ over $\mathbb{Z}$---and hence over any ring---in Example \ref{ex:delta}.
\end{ex}

\subsection{Bounded models}
Motivated by Example \ref{ex:dkc}, introduce the $\mathbb{Z}$-linear (which is to say, $\mop{Mod}_{\mathbb{Z}}$-enriched) category $\Ch$ 
$$
0 \overset{\partial_0}{\longrightarrow} 1 \overset{\partial_1}{\longrightarrow} 2 \overset{\partial_2}{\longrightarrow} 3 \overset{\partial_3}{\longrightarrow} \cdots
$$
freely generated by the morphisms $\partial_i$ subject to the relations $\partial_{i+1} \circ \partial_{i} = 0$.  By construction, a $\mathbb{Z}$-linear functor $\Ch \longrightarrow \RMod$ is a cochain complex of $k$-modules.  Finitely generated representations of $\Ch$ are ``bounded'' in the sense that they vanish away from a finite set of objects.  
The Dold-Kan correspondence says that finitely generated $\Delta$-representations are equivalent to finitely generated $\mathbb{Z}$-linear $\Ch$-representations, and so the apparently infinite data of a finitely generated $\Delta$-representation compresses to a bounded $\Ch$-representation.

In Example \ref{ex:dkc}, the existence of the category $\Ch$ immediately gives that $\Delta$ is dimension zero over any Artinian ring.  We now argue the other direction, showing that $\mathcal{D}$ is dimension zero over $k$ exactly when there exists some $k$-linear category $\mathcal{C}$ whose finitely generated representations give bounded versions of finitely generated representations of $\mathcal{D}$.  In what follows, $\catname{Rep}_k(\mathcal{D})$ stands for the category of functors $\mathcal{D} \longrightarrow \RMod$, and $\catname{Rep}_k(\mathcal{C})$ stands for the $k$-linear functors $\mathcal{C} \longrightarrow \RMod$.

\begin{defn}
A \textbf{bounded model} for $\catname{Rep}_k(\mathcal{D})$ is an equivalence of categories 
$$
\catname{Rep}_k(\mathcal{D}) \overset{\sim}{\longleftrightarrow} \catname{Rep}_{k}(\mathcal{C})
$$
for some $k$-linear category $\mathcal{C}$ with the property that every finitely generated $\mathcal{C}$-representation $V$ has $Vc = 0$ for all but finitely many objects $c \in \mathcal{C}$.
\end{defn}

\begin{thm} \label{thm:Cauchy}
A category $\mathcal{D}$ is dimension zero over $k$ if and only if $k$ is Artinian and $\catname{Rep}_k(\mathcal{D})$ has a bounded model $\catname{Rep}_k(\mathcal{D}) \overset{\sim}{\longleftrightarrow} \catname{Rep}_{k}(\mathcal{C})$ where each $\Hom_{\mathcal{C}}(p, p')$ is a finitely generated $k$-module.  
\end{thm}

It follows that if $\mathcal{D}$ is dimension zero over $k$, any finitely generated $\mathcal{D}$-representation can be described completely by giving the non-zero maps of its corresponding $\mathcal{C}$-representation.

\begin{rem} \label{rem:ordinary}
Concretely, we may take $\mathcal{C}$ to be the opposite of the full subcategory of $\catname{Rep}_k(\mathcal{D})$ spanned by one representative of each isomorphism class of indecomposable projective.  By analogy with the case of a finite dimensional algebra, this category may be termed the ``ordinary quiver'' of $\mathcal{D}$.
\end{rem}

It is worth mentioning, however, that finding the indecomposable projectives of a category (or even a monoid) is computationally expensive.  It is often much easier to find a homological modulus than a bounded model.  In the next section, we show that some of the benefits of a bounded model are already conferred by a homological modulus.

\subsection{Consequences of a homological modulus}
In this section, we assume that $\mathcal{D}$ admits a homological modulus $\mu$ over $k$, and write $\mu( \{ d_i \}_{i \in I}) = \cup_{i \in I} \mu(d_i)$ for any (possibly infinite) indexing set $I$.  We do not assume that the hom-sets $\Hom_{\mathcal{D}}(x, y)$ are finite or that $k$ is Artinian.  We also drop the assumption that each $\mu(d)$ be finite.
\noindent

Recall that the \textbf{full subcategory} spanned by a collection of objects is the largest subcategory containing only those objects.
\begin{thm} \label{thm:determined}
Representations generated in degrees $\{ d_i \}_{i \in I}$ are determined up to isomorphism by their restrictions to the full subcategory spanned by the objects $\mu( \{d_i\}_{i \in I})$.
\end{thm}

\begin{cor} \label{cor:Noetherian}
If $V$ is generated in degrees $\{ d_i \}_{i \in I}$, then any subrepresentation of $V$ is generated in degrees $\mu( \{ d_i \}_{i \in I} )$. 
\end{cor}

\noindent
In the next result, suppose $V \in \catname{Rep}_k(\mathcal{D})$ is generated in degrees $\{d_i\}_{i \in I}$, and write $\mathcal{E}_q$ for the full subcategory spanned by the collection of objects 
$$
\bigcup_{p = 0}^{q+1} \mu^p(\{d_i\}_{i \in I}),
$$
where $\mu^{p}$ denotes the $p$-fold self-composition of the homological modulus.
\begin{cor} \label{cor:exts}
For any representation $W \in \catname{Rep}_k(\mathcal{D})$, we have 
$$
\Ext^{q}_{\mathcal{D}}(V, W) \simeq \Ext^{q}_{\mathcal{E}_q}(V', W'),
$$
where $V'$ and $W'$ are the restrictions to $\mathcal{E}_q$.  If $k$ is Noetherian, $V$ is finitely generated, $W'$ takes values in finitely generated $k$-modules, and the hom-sets $\Hom_{\mathcal{E}_q}(x, y)$ are finite, then $\Ext^{q}_{\mathcal{D}}(V, W)$ is finitely generated as a $k$-module. 
\end{cor}

\subsection{Computation in the presence of a homological modulus; Effective contexts}
We present a refinement of the data of a homological modulus that (1) always exists, (2) is convenient for proofs and computations, and (3) gives the recipe for actually recovering $V$ from its restriction as in Theorem \ref{thm:determined}.
\begin{defn}
An \textbf{effective context} refining a homological modulus $\mu$ for a category $\mathcal{D}$ over a commutative ring $k$ consists of the following data.  For every pair of objects $d, x \in \mathcal{D}$:
\begin{itemize}
\item $\kappa$, a natural number
\item $m = (m_1, m_2, \ldots, m_\kappa)$, a $\kappa$-tuple of objects $m_i \in \mu(d)$
\item $\alpha = (\alpha_1, \alpha_2, \ldots, \alpha_\kappa)$, a $\kappa$-tuple of finitely-supported functions $\alpha_i \colon \Hom_{\mathcal{D}}(x, m_i) \longrightarrow k$
\item $\beta = (\beta_1, \beta_2, \ldots, \beta_\kappa)$, a $\kappa$-tuple of finitely-supported functions $\beta_i \colon \Hom_{\mathcal{D}}(m_i, x) \longrightarrow k$,
\end{itemize}
with the requirement that for every pair of arrows $f, g : d \longrightarrow x$,
\begin{equation} \label{eq:ecc}
\sum_{i=1}^{\kappa} \; \; \sum_{\substack{p \; : \; x \longrightarrow m_i \\ q \; : \; m_i \longrightarrow x \\ q \circ p \circ f = g}}  \alpha_i(p) \beta_i(q) = \left\{
     \begin{array}{lcr}
       1 & : f = g \phantom{.}\\
       0 & : f \neq g.
     \end{array}
   \right.
\end{equation}
\end{defn}
\begin{prop} \label{prop:refinement}
If $\mathcal{D}$ admits a homological modulus $\mu$, then $\mathcal{D}$ admits an effective context refining $\mu$.
\end{prop}
Let $V$ be a representation generated in degrees $\{ d_1, \ldots, d_l \}$.  We give a formula that computes any induced map $Vf: Vx \longrightarrow Vy$ using only those maps with $x, y \in \mu( \{d_1, \ldots, d_l \})$.  In what follows, use superscripts to distinguish the data of the effective context associated to the various $d_j$ and various $x$.  For example, write $\alpha^{j,x}$ for the $\alpha$ associated to $d_j$ and $x$.

\begin{cons} \label{cons:key}
Let $V$ be a $\mathcal{D}$-representation, and let $f : x \longrightarrow y$ be an arrow of $\mathcal{D}$.  We build a block matrix $X=X_V(f)$ depending on $V$ and $f$, and also a linear map $Y_V(f)$.  Given $i, j \in \{ 1, \ldots, l \}$, define a $\kappa^{i,y} \times \kappa^{j,x}$ block matrix $T^{ij}$ whose $(r, s)$-entry is given by
$$
T^{ij}_{rs} = \sum_{\substack{p \; : \; m^{j,x}_s \longrightarrow x \vspace{2pt} \\ q \;:\; y \longrightarrow m^{i,y}_r}} \beta^{j, x}_s(p) \cdot \alpha^{i, y}_r(q) \cdot V(q \circ f \circ p).
$$
Similarly, for $i, j$ with $1 \leq i < j \leq l$, define a $\kappa^{i,x} \times \kappa^{j,x}$ block matrix $U^{ij}$ whose $(r, s)$-entry is given by
$$
U^{ij}_{rs} = \sum_{\substack{p \; : \; m^{j,x}_s \longrightarrow x \vspace{2pt} \\ q \; : \; x \longrightarrow m^{i,x}_r}} \beta^{j, x}_s(p) \cdot \alpha^{i, x}_r(q) \cdot V(q \circ p).
$$
We now assemble the $l \times l$ block matrix $X$.  The entries in the first column are given by $X_{i1} = T^{i1}$.  Compute every subsequent column from the previous columns using the formula
$$
X_{ij} = T^{ij} - \sum_{n = 1}^{j-1} X_{in} \circ U^{nj}.
$$
If $z \in \mathcal{D}$, factor the linear map $X_V(1_z)$ as the composite of a surjection $B_V(z)$ and an injection $A_V(z)$ so that $X_V(1_z) = A_V(z) \circ B_V(z)$.  Now define the linear map $Y_V(f)$ by the formula
$$
Y_V(f) = B_V(y) \circ X_V(f) \circ A_V(x).
$$
\end{cons}

\begin{thm} \label{thm:main}
Suppose $\mathcal{D}$ is equipped with an effective context refining $\mu$ and that $V$ is a representation of $\mathcal{D}$ generated in degrees $\{d_1, d_2, \ldots, d_l \}$.
\begin{enumerate}
\item The linear maps $Y_V(f)$ built in Construction \ref{cons:key} depend only on the linear maps $V(g)$ where $g$ ranges over the arrows of $\mathcal{D}$ whose source and target both lie in the finite set $\mu(\{d_1, d_2, \ldots, d_l\})$.
\item The assignment $f \mapsto Y_V(f)$ defines a representation of $\mathcal{D}$.
\item There is an isomorphism of representations $V \simeq Y_V$.
\end{enumerate}
\end{thm}

\section*{Acknowledgements}
Thanks to Daniel Barter, Julian Rosen, Andrew Snowden, and David Speyer for useful conversations.  Thanks also to Jordan Ellenberg for suggesting that Theorem \ref{thm:Cauchy} might be true.  Thanks as well to math.stackexchange users Markus Scheuer and hypergeometric for verifying the claim at the end of Example \ref{ex:ptsets} via question \href{http://math.stackexchange.com/questions/1378443}{1378443}.  The author was supported by an NSF Graduate Research Fellowship (ID 2011127608).  This work contains results that later appeared in his 2016 Ph.D. thesis at the University of Michigan.  The author acknowledges support from the algebra RTG at the University of Wisconsin, NSF grant DMS-1502553.
\section{Supporting results and Examples}
\subsection{Useful properties of the relation $\leq_d$}
\begin{prop} \label{prop:invertible}
If $\Hom_{\mathcal{D}}(d, x)$ is finite, then $x \leq_d y$ if and only if the matrices $M_s$ contain an invertible matrix in their span.
\end{prop}
\begin{proof}
The matrices $M_s$ are closed under multiplication, and the $k$-algebra generated by a finite invertible matrix contains the identity by the Cayley-Hamilton theorem.
\end{proof}

\begin{rem}[Probabilistic computation of the relation $\leq_d$]
If a random linear combination of the matrices $M_s$ is invertible, then $x \leq_d y$ by Proposition \ref{prop:invertible}; otherwise, for many reasonable methods of generating random linear combinations, it is likely that $s \not \leq_d y$.  Often, a bound on the probability of a false negative can be obtained via the Schwartz-Zippel-DeMillo-Lipton lemma \cite{Schwartz}.  This gives a practical method of experimentation when attempting to apply Theorem \ref{thm:detection}.
\end{rem}

We introduce concise notation for useful submodules of the $k$-linearized hom-sets of $\mathcal{D}$.  If $d_1, d_2, \ldots, d_l \in \mathcal{D}$ are objects, write $(d_1, d_2, \ldots, d_l)$ for the free $k$-module spanned by all maps $d_1 \longrightarrow d_l$ that can be written as a composite $d_1 \longrightarrow d_2 \longrightarrow \cdots \longrightarrow d_l$.  For example, $(x, y) = k \Hom_{\mathcal{D}}(x, y)$, and $(x, y, x) = k S(x, y)$.  Write $\cdot$ for the $k$-bilinearized composition law in $\mathcal{D}$, so $(x, y) \cdot (y, z) = (x, y, z)$, for example.  If $N \subseteq (x, y)$ is a submodule, define left and right annihilator submodules of $N$
\begin{align*}
\lann{N}{d} & = \{ f \in N : f \cdot (y, d) = 0 \} \\
\rann{d}{N} & = \{ f \in N : (d, x) \cdot f = 0 \}.
\end{align*}
We shall find useful the inclusion $(w,x) \cdot \rann{d}{(x,y)} \cdot (y, z) \subseteq \rann{d}{(w, x, y, z)}$.

\begin{prop}[Alternative characterization of the relation $\leq_d$]\label{prop:alt}
With the notation above, $x \leq_d y$ exactly when $\rann{d}{(x, x)} + (x, y, x) = (x, x)$.
\end{prop}
\begin{proof}
The matrices $M_s$ with $s \in S(x, y)$ give the action of $(x, y, x) \subseteq (x, x)$ on $(d, x)$ by post-composition.  Their span is a (possibly unitless) $k$-algebra isomorphic to $(x, y, x)/((x, y, x) \cap \rann{d}{(x, x)})$, by construction.  This algebra has a unit exactly when $1 \in \rann{d}{(x, x)} + (x, y, x)$, and the result follows since this sum is a two-sided ideal of $(x, x)$ containing $1$.
\end{proof}

\begin{obs}[Retracts and the relation $\leq_d$] \label{obs:retract}
If $x$ is a retract of $y$, then $x \leq_d y$ for any $d$.  If $x \leq_x y$, then $x$ is a retract of $y$.
\end{obs}
\begin{proof}
If $x$ is a retract of $y$, then $1 \in S(x, y)$, and $M_1 = 1$.  Alternatively, $(x, y, x) \subseteq (x, x)$ is a two-sided ideal containing $1$, and so $(x, y, x) = (x, x)$.  Similarly, if $x \leq_x y$, then $(x, y, x) = (x, x)$ since $\rann{x}{(x, x)} = 0$.
\end{proof}

\begin{prop} \label{lem:preorder}
Each relation $\leq_d$ is reflexive and transitive, and so defines a preorder. 
\end{prop}
\begin{proof}
We have $x \leq_d x$ by Observation \ref{obs:retract} since any $x$ is a retract of itself.  Suppose $x \leq_d y$ and $y \leq_d z$.  Using Proposition \ref{prop:alt}, we get two equations
\begin{align*}
\rann{d}{(x, x)} + (x, y, x) &= (x, x) \\
\rann{d}{(y, y)} + (y, z, y) &= (y, y).
\end{align*}
Multiplying the second equation by $(x, y)$ on the left and $(y, x)$ on the right,
\begin{align*}
(x, y) \cdot \rann{d}{(y, y)} \cdot (y, x) + (x, y) \cdot (y, z, y) \cdot (y, x) & = (x, y) \cdot (y, y) \cdot (y, x) \\
 (x, y) \cdot \rann{d}{(y, y)} \cdot (y, x) + (x, y, z, y, x) & = (x, y, x).
\end{align*}
Substituting this formula for $(x, y, x)$ into the first equation,
$$
\rann{d}{(x, x)} +  (x, y) \cdot \rann{d}{(y, y)} \cdot (y, x) + (x, y, z, y, x) = (x, x).
$$
We have a chain of inclusions
\begin{align*}
(x, x) & = \rann{d}{(x, x)} +  (x, y) \cdot \rann{d}{(y, y)} \cdot (y, x) + (x, y, z, y, x) \\
& \subseteq \rann{d}{(x, x)} + \rann{d}{(x, y, x)} + (x, z, x) \\
& \subseteq \rann{d}{(x, x)} + \rann{d}{(x, x)} + (x, z, x) \\
& \subseteq \rann{d}{(x, x)} + (x, z, x) \\
& \subseteq (x, x),
\end{align*}
proving that $\rann{d}{(x, x)} + (x, z, x) = (x,x)$, and $x \leq_d z$ by Proposition \ref{prop:alt}.
\end{proof}

\begin{defn}
Write $x \leq^d y$ whenever $x \leq_d y$ in the opposite category $\mathcal{D}^{op}$.
\end{defn}

\begin{prop} \label{prop:monoann}
If $c \leq^x d$, then $\rann{d}{(x,y)} \subseteq \rann{c}{(x,y)}$.
\end{prop}
\begin{proof}
By Proposition \ref{prop:alt}, $\lann{(c, c)}{x} + (c, d, c) = (c, c)$. Let $f \in (x, y)$, and suppose $(d, x) \cdot f = 0$.  Then
\begin{align*} 
(c, x) \cdot f & = (c, c) \cdot (c, x) \cdot f \\
& = \lbrack \lann{(c, c)}{x} + (c, d, c) \rbrack \cdot (c, x) \cdot f \\
& = \lann{(c, c)}{x} \cdot (c, x) \cdot f + (c, d, c) \cdot (c, x) \cdot f \\
& = 0 + (c, d, c, x) \cdot f \\
& = (c, d) \cdot (d, c, x) \cdot f \\
& \subseteq (c, d) \cdot (d, x) \cdot f \\
& \subseteq 0,
\end{align*}
and the claim is proved.
\end{proof}

\begin{lem} \label{lem:loweringsubscripts}
If $x \leq_d y$ and $c \leq^x d$, then $x \leq_{c} y$.
\end{lem}
\begin{proof}
By Proposition \ref{prop:monoann}, $\rann{d}{(x,x)} \subseteq \rann{c}{(x,x)}$, and by Proposition \ref{prop:alt}, $\rann{d}{(x, x)} + (x, y, x) = (x, x)$.  So
$$
(x, x) =  \rann{d}{(x, x)} + (x, y, x) \subseteq \rann{c}{(x, x)} + (x, y, x) \subseteq (x, x)
$$
and $x \leq_c y$ as required.
\end{proof}
\subsection{Examples: dimension zero categories, homological moduli, effective contexts}
\begin{ex} (Homological modulus for $\Delta$) \label{ex:delta}
Since $\Delta$ is dimension zero over any Artinian ring by the Dold-Kan correspondence, it comes as no surprise that $\Delta$ admits a homological modulus over $\mathbb{Z}$.  We show that one may take $\mu(d) = \{ d+1 \}$.  In a moment, we shall give a direct argument that for all $n \in \Delta$, 
$$n+2 \leq_n n+1.$$
Assuming this lemma, the following deductions finish the proof.  Since $d$ is a retract of any $n \geq d$, we know by Observation \ref{obs:retract} that $d \leq^{n+2} n$, from which it follows by Lemma \ref{lem:loweringsubscripts} that $n + 2 \leq_{d} n + 1$.  By Lemma \ref{lem:preorder}, $\leq_{d}$ is a transitive relation, and so $d+1$ is a maximum for the preorder $\leq_{d}$.

We now seek to prove that $n+2 \leq_n n+1.$  Define the set
$$
H = \left\{ h \in \Hom_{\Delta}(n+2, n+2) : i \leq h(i) \leq i+1 \mbox{ for all $i$} \right\},
$$
and for each $k \in \{1, \ldots, n+1\}$, let $\iota_k$ be the involution of $H$ so that $\iota_k(h)(i) = h(i)$ when $i \neq k$ and $\iota_k(h)(k) = 2k + 1 - h(k)$.  In other words, $\iota_k$ flip-flops function values at $k$ between $k$ and $k+1$.  Define a function $\rho : H \longrightarrow \mathbb{Z}$, $\rho(x) = \prod_{i = 1}^{n+2} (-1)^{h(i)-i}$.  It is immediate that each $\iota_k$ satisfies $\rho(\iota_k(x)) = -\rho(x)$.  Let $f, g \in \Hom_{\Delta}(n, n+2)$, and consider the sum
$$
\sigma = \sum_{\substack{h \in H \\ h \circ f = g}} \rho(h).
$$
Choosing any $k \not \in \im(f)$, $k \neq n+2$, we see that $h \circ f = \iota_k(h) \circ f$, and so $\iota_k$ permutes the terms of the sum, leaving $\sigma$ invariant.  On the other hand, each term of $\sigma$ is anti-invariant under $\iota_k$.  It follows that $\sigma = 0$.  Noting that every element of $H$ factors through $n+1$ except for the identity function on $n+2$, we see that the identity matrix shows up exactly once in $\sum_{h \in H} \rho(h) M_h = 0$, and every other term has $h \in S(n+2, n+1)$.  This proves that the identity matrix is in the span of the $M_s$ with $s \in S(n+2, n+1)$, and so $n+2 \leq_n n+1.$
\end{ex}

\begin{ex}[The category $\Fin_*$ is dimension zero over any Artinian ring]
By a result of Pirashvili \cite[Theorem 3.1]{Gamma}, representations of the category of finite pointed sets $\Fin_*$ that send $\ast$ to zero are equivalent to representations of $\Omega$, the category of nonempty finite sets with surjections.  Since maps in $\Omega$ go from larger sets to smaller sets, finitely generated $\Omega$-representations are finite length.  That $\Fin_*$ is dimension zero over any Artinian ring is not much harder than Pirashvili's result.  We give a different proof in the next example.
\end{ex}

\begin{ex}[Effective context for $\Fin_*$]\label{ex:ptsets}
Use the notation $n_* = \{*, 1, 2, \ldots, n\}$ for a finite pointed set with $n$ elements along with a basepoint.  We give an effective context refining the homological modulus $\mu(d_*) = \{0_*, 1_*, \ldots, d_*\}$.  Set $\kappa^{d_*,n_*} = \sum_{i = 0}^d \binom{n}{i} $ and $m^{d_*,n_*} = (0_*, 1_*, 1_* \ldots, 1_*, 2_*,  \ldots, 2_*, \ldots, d_*, \ldots, d_*)$, the tuple of pointed subsets of $n_*$ with up to $d$ elements besides the basepoint.  If $p_* \subseteq n_*$ is a pointed subset of $n_*$, set $\alpha^{d_*, p_*}(f) = (-1)^{d-p}$ if $f$ is the unique surjection $n_* \twoheadrightarrow p_*$ restricting to the identity on $p_*$ and sending every other element to the basepoint, and $\alpha^{d_*, p_*}(f) =0$ otherwise.  Similarly, set $\beta^{d_*, p_*}(g) = \binom{n-p-1}{d-p}$ if $g$ is the natural injection $p_* \hookrightarrow n_*$, and $0$ otherwise.  One may verify that (\ref{eq:ecc}) holds, and so these choices give an effective context.
\end{ex}

\begin{ex}[The category $\catname{FI\#}$ is dimension zero over any Artinian ring]
The category $\catname{FI\#} \subseteq \Fin_*$ of pointed set maps injective away from the basepoint was introduced by Church, Ellenberg, and Farb, who show that it has a bounded model given by the $k$-linearization of the category of finite sets and bijections \cite[Theorem 4.1.5]{CEF}.  It follows that $\catname{FI\#}$ is dimension zero over any commutative Artinian ring.
\end{ex}

\begin{ex}[Effective context for $\catname{FI\#}$]
The effective context produced for $\Fin_*$ in Example \ref{ex:ptsets} is supported on the subcategory of functions that are injective away from the basepoint, so the effective context given earlier for $\Fin_*$ works equally well for $\catname{FI\#}$.
\end{ex}

\begin{ex}[A homological modulus for $\Fin$]\label{ex:finfl}
The paper \cite{UniformlyPresentedVectorSpaces} shows that, over the rationals, $\Fin$ has a homological modulus given by $\mu(0) = \{0, 1 \}$, $\mu(d) = \{d+1\}$ for $d \neq 0$ (although the results of the present paper allow us to deduce this result for any field from Example \ref{ex:delta}).  Over $\mathbb{Q}$, a basic projective splits as a direct sum $\mathbb{Q} \Hom_{\Fin}(k, -) = \Lambda^k \oplus \Theta^k$, where $\Lambda^k$ is the anti-invariants for the $S_k$-action by pre-composition, and $\Theta^k$ is a complementary representation.  Any nonzero representation admits a nonzero map from one of these representations, so the opposite of the full subcategory spanned by these projectives has the same representation theory as $\Fin$.  Using this notation, the paper \cite{UniformlyPresentedVectorSpaces} also gives the more precise result that one may take $\mu(\Lambda^k) = \{\Lambda^k, \Lambda^{k+1}\}$ and $\mu(\Theta^k) = \{\Lambda^k, \Theta^k\}$.
\end{ex}

\begin{ex}[Noncommutative finite sets are dimension zero]
In \cite{JSEJWG}, Pirashvili-Richter's category of noncommutative finite sets \cite{HochFunctorHomology} is shown to be dimension zero over any field, and the simples in characteristic zero are deduced from \cite{UniformlyPresentedVectorSpaces} (these simples appeared earlier in \cite{rains}).
\end{ex}

\begin{ex}[The category $\catname{Rel}$ of finite sets with relations is dimension zero]
Andrew Gitlin has produced a homological modulus over $\mathbb{Q}$ for the category $\catname{Rel}$ of finite sets with relations, wherein $\mu(d) = \{2^d\}$.  This is the first example known to the author of a homological modulus that is not ``linear in $d$.''  The irreducible representations over $\mathbb{Q}$ have now been constructed by Serge Bouc and Jacques Th\'evenaz \cite{bouc}.
\end{ex}

\begin{ex}[Homological modulus for $\catname{Vect}_{\mathbb{F}_q}$]
The category of finite dimensional vector spaces over a finite field of characteristic $p$ admits a homological modulus over $\mathbb{Z}[\frac{1}{p}]$ by work of Kuhn \cite{KuhnIdempotents}, who relies on idempotents constructed much earlier by Kov{\'a}cs  \cite{KovacsIdempotents}.  Working in a skeleton for $\catname{Vect}_{\mathbb{F}_q}$ where objects are natural numbers and morphisms are $\mathbb{F}_q$-matrices, we may take $\mu(n) = \{ n \}$.  Indeed, the two-sided ideal $(n+1, n, n+1) \subseteq ( n+1, n+1)$ is generated by an explicit idempotent that acts as an identity for this ideal considered as a subalgebra, and so witnesses $n + 1 \leq_n n$.  Imitating the argument given at the beginning of Example \ref{ex:delta} finishes the proof.
\end{ex}

\begin{ex}[Connection to $\catname{FI}$-modules and Noetherian categories]
A category $\mathcal{D}$ is called Noetherian over a commutative ring $k$ if subrepresentations of finitely generated representations are finitely generated.  Certainly any dimension zero category is also Noetherian.  A nice example of a Noetherian category is the category of finite sets with injections, $\catname{FI}$.  This category was shown to be Noetherian over $\mathbb{Q}$ in \cite{CEF}, and later over any Noetherian ring in \cite{CEFN}.  Over the rationals, Steven Sam and Andrew Snowden \cite[Corollary 2.2.6]{GLEquivariantInfinite} prove that this category has dimension one: after modding out by the Serre subcategory of finite length representations, the remaining representations have finite length.  In \cite{SSNoetherian}, they provide Gr\"obner-style techniques for proving categories are Noetherian.  They are also able to recover the result 
that $\Fin$ is dimension zero over any commutative Artinian ring. 
\end{ex}


\section{Proofs}
\subsection{Proof of Theorem \ref{thm:main}} For each $z \in \mathcal{D}$, we shall construct a particular choice of matrices $A(z), B(z)$ satisfying $B(z) \circ A(z) = 1$ so that the block matrix $X$ from Construction~\ref{cons:key} factors 
$$X(f) = A(y) \circ V(f) \circ B(x).$$
This is enough to prove the result since epi-mono factorizations are unique up to isomorphism.  For every $i \in \{1, \ldots, l\}$ and $z \in \mathcal{D}$, define $C^i(z)$, a $\kappa^{i,z} \times 1$ block matrix with $(r, 1)$ entry
$$
C^i(z)_{r1} = \sum_{p \; : \; z \longrightarrow m^{i,z}_r} \alpha^{i,z}_r(p) \cdot V(p).
$$
Similarly, for every $j \in \{1, \ldots, l\}$ and $z \in \mathcal{D}$, define $D^i(z)$, a $1 \times \kappa^{j,z}$ block matrix whose $(1, s)$ entry is
$$
D^j(z)_{1s} = \sum_{q \; : \; m^{j,z}_s \longrightarrow z} \beta^{j,z}_s(q) \cdot V(q).
$$
Define an $l \times 1$ block matrix $A(z)$ so its $(i,1)$ entry is $A(z)_{i1} = C^i(z)$.  Define a $1 \times l$ block matrix $B(z)$ inductively with $B(z)_{11} = D^1(z)$ and subsequent entries given by
$$
B(z)_{1j} = D^j(z) - \sum_{n = 1}^{j-1} B(z)_{1n} \circ C^n(z) \circ D^j(z).
$$
It is straight-forward that $X(f) = A(y) \circ V(f) \circ B(x)$.  We show $B(z) \circ A(z) = 1$ holds by induction on $l$.  Let us establish the claim for $l=1$.  In this case, $V$ is generated in degrees $\{ d_1 \}$ (a singleton set) and so it suffices to prove $B(z) \circ A(z) \circ V(h) = V(h)$ for all $z$ and all $h:d_1 \longrightarrow z$.
\begin{align*}
B(z) \circ A(z) \circ V(h) &=  D^1(z) \circ C^1(z)  \circ V(h) \\
&= \left( \sum_{r = 1}^{\kappa^{1,z}}  \;\; \sum_{\substack{p \; : \; z \longrightarrow m_r \\ q \; : \; m_r \longrightarrow x }} \beta^{1,z}_r(q) \cdot  \alpha^{1,z}_r(p) \cdot V(q \circ p)   \right) \circ V(h) \\
&= \sum_{r = 1}^{\kappa^{1,z}}  \;\; \sum_{\substack{p \; : \; z \longrightarrow m_r \\ q \; : \; m_r \longrightarrow z }}   \alpha^{1,z}_r(p) \cdot \beta^{1,z}_r(q) \cdot V(q \circ p \circ h) \\
&=  \sum_{r = 1}^{\kappa^{1,z}} \;\; \sum_{g\; : \; d_1 \longrightarrow z} \;\; \sum_{\substack{p \; : \; z \longrightarrow m_r \\ q \; : \; m_r \longrightarrow z \\ q \circ p \circ h = g}} \alpha^{1,z}_r(p) \cdot \beta^{1,z}_r(q) \cdot V(g) \\
&= \sum_{g\; : \; d_1 \longrightarrow z} \;\; \left( \sum_{r = 1}^{\kappa^{1,z}}   \;\; \sum_{\substack{p \; : \; z \longrightarrow m_r \\ q \; : \; m_r \longrightarrow z \\ q \circ p \circ h = g}} \alpha^{1,z}_r(p) \cdot \beta^{1,z}_r(q) \right) \cdot V(g) \\
&= V(h).
\end{align*}
Assume the result for $l$; we proceed with the case of $l+1$, for which there is a new object $d_{l+1}$ at the end of the list of generator degrees.  Performing the construction again, we obtain new matrices $A^+(z)$ and $B^+(z)$ that are slightly larger than the previous matrices $A(z)$ and $B(z)$.  By construction, the final entry of $B^+(z)$ can be expressed using the previous entries:
\begin{align*}
B^+(z)_{1,(l+1)} &=  D^{l+1}(z) - \sum_{n = 1}^{l} B(z)_{1n} \circ C^n(z) \circ D^{l+1}(z) \\
&= D^{l+1}(z) - B(z) \circ A(z) \circ D^{l+1}(z) \\
&= \left[ 1 - B(z) \circ A(z) \right] \circ  D^{l+1}(z).
\end{align*}
Consequently,
\begin{align*}
B^+(z) \circ A^+(z) &= B(z) \circ A(z) + B^+(z)_{1,(l+1)} \circ A^+(z)_{(l+1),1} \\
&= B(z) \circ A(z) + \left[1 - B(z) \circ A(z) \right] \circ  D^{l+1}(z) \circ C^{l+1}(z).
\end{align*}
Given any map $h : d_i \longrightarrow z$ with $i \leq l$, 
\begin{align*}
B^+(z) \circ A^+(z) \circ V(h) &= B(z) \circ A(z) \circ V(h) + \left[1 - B(z) \circ A(z) \right] \circ  D^{l+1}(z) \circ C^{l+1}(z) \circ V(h) \\
&= V(h) + \left[1 - B(z) \circ A(z) \right] \circ \left( \,\sum_{g \; : \; d_i \longrightarrow z} \gamma_g  \cdot V(g) \right)\\
&= V(h),
\end{align*}
where the coefficients $\gamma_g$ expanding $D^{l+1}(z) \circ C^{l+1}(z) \circ V(h)$ do not matter since every term in the sum is anyhow annihilated by $\left[1 - B(z) \circ A(z) \right]$ using the inductive hypothesis.  Given instead a map $h : d_{l+1} \longrightarrow z$, then a reprise of the computation for the base case gives $D^{l+1}(z) \circ C^{l+1}(z) \circ V(h) = V(h)$ and so
\begin{align*}
B^+(z) \circ A^+(z) \circ V(h) &= B(z) \circ A(z) \circ V(h) + \left[1 - B(z) \circ A(z) \right] \circ  D^{l+1}(z) \circ C^{l+1}(z) \circ V(h) \\
&= B(z) \circ A(z) \circ V(h) + \left[1 - B(z) \circ A(z) \right] \circ V(h)\\
&= V(h).
\end{align*}

\subsection{Proof of Proposition \ref{prop:refinement}, Theorem \ref{thm:determined}}
\begin{proof}
Let us begin by showing that any homological modulus $\mu$ can be refined to an effective context.  For any pair of objects $x, y \in \mathcal{D}$, we have a surjection $\psi^{x,y} : \Hom_{\mathcal{D}}(x, y) \times \Hom_{\mathcal{D}}(y, x) \longrightarrow S(x, y)$ sending a pair of arrows to its composition (which is necessarily a self-map of $x$ that factors through $y$.)  Choose sections $\sigma^{x,y}$ satisfying $\psi^{x,y} \circ \sigma^{x,y} = 1$.  By definition, a statement of the form $x \leq_d y$ provides ring elements $\rho(s) \in k$ with $\sum_{s \in S(x,y)} \rho(s) \cdot M_s = 1$; our sections give 
$$
\sum_{s \in S(x,y)} \;\; \sum_{\substack{p \; : \; x \longrightarrow y \\ q \; : \; y \longrightarrow x \\ q \circ p = s}} \alpha_s(p) \cdot \beta_s(q) \cdot M_{s} = 1,
$$
taking $\alpha_s(p) =
\left\{
     \begin{array}{lr}
       \rho(s) & : \sigma^{x,y}(s)_1 = p \\
       0 & : \sigma^{x,y}(s)_1 \neq p
     \end{array}
   \right.$ and  $\beta_s(q) =
\left\{
     \begin{array}{lr}
       1 & : \sigma^{x,y}(s)_2 = q \phantom{.} \\
       0 & : \sigma^{x,y}(s)_2 \neq q.
     \end{array}
   \right.$

The only nonzero terms in the sum occur when $\rho(s) \neq 0$ which happens only finitely often; taking $\kappa = \# \{s : \rho(s) \neq 0 \}$ completes the proof of Proposition \ref{prop:refinement}. 

A standard application of Zorn's lemma to build the isomorphism deduces Theorem \ref{thm:determined} from Proposition \ref{prop:refinement} and Theorem \ref{thm:main}.

\subsection{Proof of Corollaries \ref{cor:Noetherian} and \ref{cor:exts}}

If $V$ is generated in degrees $\{ d_i \}_{i \in I}$, we now show that any subrepresentation $W \hookrightarrow V$ must be generated in degrees $\mu( \{ d_i \}_{i \in I} )$.  For any $z \in \mathcal{D}$, the matrices $A(z)$ and $B(z)$ associated to $V$ that were built at the outset of the proof of Theorem \ref{thm:main} are block matrices filled with linear combinations of various induced maps $V(p)$.  Write $A'(z)$ and $B'(z)$ for corresponding block matrices that use the same block structure and coefficients, but $W(p)$ in place of $V(p)$; these matrices are designed to commute with the inclusion $i : W \hookrightarrow V$.  We have
$$
i \circ B'(z) \circ A'(z) = B(z) \circ A(z) \circ i = 1 \circ i = i \circ 1,
$$
and so $B'(z) \circ A'(z) = 1$ by cancellation of the monomorphism $i$.  It follows immediately that $\im (B'(x)) = Wx$, proving that $W$ is generated in degrees $\mu( \{ d_i \}_{i \in I} )$.  Since the kernel of a map from a representation generated in finitely many degrees is again generated in finitely many degrees, such representations form an abelian subcategory.  The proof of Corollary~\ref{cor:Noetherian} is complete.

Computation of $\Ext$ in the category of representations is aided by an explicit collection of enough projectives: the $k$-linearized representable functors.  In detail, for any $d \in \mathcal{D}$ we have a representable functor $\Hom_{\mathcal{D}}(d, -) \colon \mathcal{D} \to \catname{Set}$.  Post-composing with the free $k$-module functor, we obtain a representation $P^d$ satisfying
$$
P^d (x) = k \Hom_{\mathcal{D}}(d, x).
$$
Yoneda's lemma gives that $\Hom(P^d, V) \simeq V(d)$ for any representation $V$.  This also shows that $P^d$ is projective, since the functor $\Hom(P^d, -)$ is evaluation at $d$, which is exact.  Finally, a representation is generated in degree $d$ if and only if it is a quotient of a direct sum of copies of $P^d$.  These observations may be summarized by saying that $P^d$ is the representation freely generated by a vector in degree $d$---specifically, $1_d \in k \Hom_{\mathcal{D}}(d, d) = P^d (d)$.

Since $V$ is generated in degrees $\{ d_i \}_{i \in I}$, it is a quotient of $P_0 = \oplus_{i \in I} (P^{d_i})^{\oplus t_i}$, where the $t_i$ are cardinals indexing the generators in degree $d_i$.  The kernel of the natural surjection $P_0 \to V$ is generated in degrees $\mu(\{ d_i \}_{i \in I})$ by Corollary~\ref{cor:Noetherian} since it is a subrepresentation of $P_0$, which is generated in degrees $\{ d_i \}_{i \in I}$.  Iterating this construction, we obtain a projective resolution
$$
\cdots \longrightarrow P_2 \longrightarrow P_1 \longrightarrow P_0 \longrightarrow V \longrightarrow 0,
$$
where each $P_p$ is a direct sum of various $P^d$ with $d \in \mu^{p}(\{ d_i \}_{i \in I})$.  When $p \leq q + 1$, every such $d$ is in the subcategory $\mathcal{E}_q$.  Since $\mathcal{E}_q$ is full, $\Hom_{\mathcal{E}_q}(d, x) = \Hom_{\mathcal{D}}(d, x)$ for any $d, x \in \mathcal{E}_q$, and so $P^d$ is still projective when restricted to the subcategory $\mathcal{E}_q$ since it coincides with the $k$-linearized representable $k \Hom_{\mathcal{E}_q}(d, -)$.  We conclude that the first $k+1$ steps of the resolution $P_{\bullet} \to V$, when restricted to $\mathcal{E}_q$, are once again the first $k+1$ steps of a projective resolution.

By Yoneda's lemma, the cochain complex computing $\Ext^{q}_{\mathcal{D}}(V, W)$ using the original resolution coincides with the one computing $\Ext^{q}_{\mathcal{E}_q}(V', W')$ using the restricted resolution, and so these two $k$-modules are isomorphic.

If $k$ is Noetherian, $V$ is finitely generated, $W'$ takes values in finitely generated $k$-modules, and the hom-sets $\Hom_{\mathcal{E}_q}(x, y)$ are finite, then the cochain complex consists of finitely generated $k$-modules and so $\Ext^{q}_{\mathcal{D}}(V, W)$ is finitely generated as a $k$-module, proving Corollary~\ref{cor:exts}.
\end{proof}

\subsection{Proof of Theorems \ref{thm:detection} and \ref{thm:Cauchy}}
\begin{proof}
One direction of Theorem \ref{thm:detection} is a consequence of Corollary \ref{cor:Noetherian}: if $\mathcal{D}$ admits a homological modulus over $k$, then any subrepresentation of a representation $V$ generated in degrees $\{ d_i \}_{i \in I}$ is generated in degrees $\mu( \{ d_i \}_{i \in I})$.  It follows that the length of any such $V$ cannot exceed the length of the module $\bigoplus_{m \in \mu( \{ d_i \}_{i \in I})} Vm$.  Since $V$ is finitely generated and $\Hom_{\mathcal{D}}(x, y)$ is always finite, each $Vm$ is finitely generated, and hence finite length since $k$ is Artinian.

In the other direction, assume that $\mathcal{D}$ is dimension zero over $k$.  We first argue that every monoid algebra $k \End_{\mathcal{D}}(d)$ is Artinian, from which it follows by the result of Zel$'$manov \cite{Zelmanov} that $k$ is Artinian and $\End_{\mathcal{D}}(d)$ is finite (see Okni{\'n}ski's book \cite[p. 172, Theorem 23]{SemigroupAlgebras} for an English account).  Next we argue that every hom-set $\Hom_{\mathcal{D}}(x, y)$ is finite as well.  Finally, we use the theory of projective covers to construct a homological modulus for $\mathcal{D}$.

Let $W$ be a representation of $\End_{\mathcal{D}}(d)$.  Since $\mathcal{D}$ is small and $\RMod$ is cocomplete, we may left Kan extend $W$ from $\End(d)$ to $\mathcal{D}$, getting a representation $V = \Lan W$.  Since $\End(d)$ is a full subcategory of $\mathcal{D}$, $V$ satisfies $\Res V = W$.  Taking $W = k \End(d)$ the regular representation, we see that $\Lan W = \Lan_d^{\mathcal{D}} k = P^d$, the basic projective $\mathcal{D}$-representation freely generated by a single vector in degree $d$.  Any strictly increasing chain of subrepresentations of $k \End(d)$ gives rise to a corresponding chain in $P^d$, which must be finite by assumption.  It follows that each monoid algebra $k \End(d)$ is Artinian, and so $k$ is Artinian and $\End(d)$ is finite by the result of Zel$'$manov.

A similar argument gives that every $k\Hom_{\mathcal{D}}(x, y)$ is Artinian as a right $k\End_{\mathcal{D}}(y)$-module.  Since $\End_{\mathcal{D}}(y)$ is finite, the orbits of its action on $\Hom_{\mathcal{D}}(x, y)$ are finite as well, and generate disjoint submodules of $k\Hom_{\mathcal{D}}(x, y)$; it follows that $\Hom_{\mathcal{D}}(x, y)$ is finite, as required.  As a consequence, if $P$ is a finitely generated projective $\mathcal{D}$-representation, then $\End(P)$ is Artinian: indeed, $P$ is a summand of some finite sum of basic projectives $\oplus_d P^d$, and $\End(\oplus_d P^d) = \oplus_{d, d'} k \Hom_{\mathcal{D}}(d', d)$ by Yoneda's lemma.  

According to Krause \cite{krausekrull}, an additive category is called \textbf{Krull-Schmidt} if every object is a finite direct sum of objects having local endomorphism rings.  We have the following characterization of such categories:
\begin{lem}\cite[Cor. 4.4]{krausekrull} \label{lem:semiperfect}
An additive category is Krull-Schmidt if and only if all idempotents split and all endomorphism rings are semiperfect.
\end{lem}
Since Artinian rings are semiperfect, the category of finitely generated projective $\mathcal{D}$-representations is Krull-Schmidt.  It follows from \cite[Lem. 3.6]{krausekrull} that every simple $\mathcal{D}$-representation $S$ has an indecomposable projective cover $P$ admitting an essential surjection $P \longrightarrow S$.  In particular, any nonzero quotient of $P$ has $S$ as a composition factor.

For every $d \in \mathcal{D}$, write $\mathscr{C}^d$ for the finite set of (necessarily indecomposable) projective covers of simple representations appearing as composition factors of $P^d$.  Similarly, write $\mathscr{T}^d$ for the finite set of indecomposable projectives that are summands of $P^d$.  Call a subset of indecomposable projectives \textbf{attainable} if it appears as a subset of $\mathscr{T}^d$ for some $d$.

Fix $d \in \mathcal{D}$.  For every attainable subset $Q \subseteq \mathscr{C}^d$, choose an object $y_Q \in \mathcal{D}$ so the projective $P^{y_Q}$ witnesses the attainability of $Q$:
$$
Q \subseteq \mathscr{T}^{y_Q}.
$$
We claim that the set $\left\{ y_Q \right\}$ satisfies the property required of $\mu(d)$, the ``joint maximum'' for the preorder $\leq_d$.  Specifically, given any $x \in \mathcal{D}$, we claim $x \leq_d y_Q$ with $Q = \mathscr{T}^x \cap \mathscr{C}^d$.  In what follows, write $y = y_Q$.

The indecomposable summands of $P^x$ come in two kinds: those appearing as projective covers of simple constituents of $P^d$, and those not appearing.  In other words, we may write $P^x = E \oplus Z$ where $E$ is a sum of projectives in the set $\mathscr{T}^x \cap \mathscr{C}^d $ and $Z$ is a sum of projectives in the set $\mathscr{T}^x \cap \overline{\mathscr{C}^d}$.  Indecomposable projectives in this second set are projective covers for simples not appearing as composition factors of $P^d$, so every map $Z \longrightarrow P^d$ is zero.  On the other hand, $\mathscr{T}^x \cap \mathscr{C}^d  = Q \subseteq \mathscr{T}^y $, and so $E$ is a summand of some finite direct sum $(P^y)^{\oplus k}$.

Note two properties of $\pi_E$, projection onto the $E$ summand of $P^x$.  First, $\pi_E$ factors
$$
\pi_E : P^x \longrightarrow E \longrightarrow (P^y)^{\oplus k} \longrightarrow E \longrightarrow P^x,
$$
and second, since $\Hom(Z, P^d) = 0$, any map $\phi : P^x \longrightarrow P^d$ satisfies $\phi \circ (1 - \pi_E) = 0$.  
By Yoneda's lemma, $\pi_E$ corresponds to some element $\omega \in (x, x)$.  The first property of $\pi_E$ ensures that $\omega \in (x, y, x)$, and the second ensures that $(1 - \omega) \in \rann{d}{(x, x)}$; it follows that $1$ is in the sum of these two ideals, and by Proposition \ref{prop:alt}, $x \leq_d y$.

We now set about proving Theorem \ref{thm:Cauchy}.  Certainly if $\Rep$ has a bounded model then its finitely generated representations have finite length.  We now show that the category $\mathcal{C}$ from Remark \ref{rem:ordinary}, is bounded and has the same representation theory as $\mathcal{D}$.  Recall that $\mathcal{C}$ was defined to be the $k$-linear category opposite to the full subcategory of $\Rep$ spanned by one representative of each indecomposable projective.

Appealing to Theorem \ref{thm:detection}, we know that $\mathcal{D}$ has finite hom-sets and that the indecomposable projective covers of simple $\mathcal{D}$-representations are enough projectives.  This gives the equivalence $\catname{Rep}_k(\mathcal{D}) \overset{\sim}{\longleftrightarrow} \catname{Rep}_{k}(\mathcal{C})$.  Similarly, we know that the basic projectives $P^d$ are finite direct sums of indecomposable projectives, and hence the $k$-module of morphisms between two indecomposable projectives $\Hom(p, p')$ is finitely generated.

In order to prove that every finitely generated $\mathcal{C}$-representation vanishes away from finitely many objects, it suffices to show that the basic projectives $\Hom_{\mathcal{C}}(p, -)$ are so bounded.  It is immediate that any particular indecomposable projective $P$ only receives nonzero maps from finitely many other indecomposable projectives $P'$, since the image of such a map must have the unique simple quotient of $P'$ as a composition factor, and $P$ only has finitely many composition factors.  It follows that in the opposite category, $\Hom_{\mathcal{C}}(p, p') = 0$ for all but finitely many $p'$, and so the claim is proved.
\end{proof}
\bibliographystyle{alphaurl}
\bibliography{references}
\end{document}